\documentclass{amsart}
\usepackage[T1]{fontenc}   

\usepackage{amssymb}
\usepackage{amssymb,amsthm} 
\usepackage{amscd,amssymb,graphics}

\usepackage[hidelinks]{hyperref}

\usepackage{amsfonts}
\usepackage{amsmath}
\usepackage{amsxtra}
\usepackage{latexsym}
\usepackage[mathcal]{eucal}
\usepackage{enumerate}
\usepackage{ifsym}
\usepackage{yfonts}
\usepackage{calc}
\usepackage{MnSymbol}
\usepackage[numbers]{natbib}

\newtheorem{thm}{Theorem}[section]
\newtheorem*{thm*}{Theorem}
\newtheorem{corol}[thm]{Corollary}
\newtheorem{claim}[thm]{Claim}
\newtheorem{lemma}[thm]{Lemma}
\newtheorem{prop}[thm]{Proposition}

\theoremstyle{definition}
\newtheorem{defi}[thm]{Definition}
\theoremstyle{remark}
\newtheorem{remark}[thm]{Remark}
\newtheorem{example}[thm]{Example}

\numberwithin{equation}{section}

\def \Z {{\Bbb Z}}

\def \seq {\subseteq}

\def \copt {\tau_{co}}
\def \eps {\varepsilon}

\def\N {{\mathbb N}}
\def\Z {{\mathbb Z}}
\def\Q {{\mathbb Q}}

\def\Aut{{\mathrm {Aut}}}
\def\Inn{{\mathrm {Inn}}}

\def\H{{\mathrm{H}}\,}

\def\co{{\,{:}\,}}

\begin{document}

\vspace{0.5in}

\renewcommand{\sc}{\scshape}
\vspace{0.5in}

\title{Minimality of the Semidirect Product}

\author{Michael Megrelishvili}
\email{megereli@math.biu.ac.il}

\author{Luie Polev}
\email{luiepolev@gmail.com}

\author{Menachem Shlossberg}
\email{shlosbm@macs.biu.ac.il}
\address{Department of Mathematics;	Bar-Ilan University; 52900 Ramat-Gan, Israel}





\begin{abstract}
	A topological group is minimal if it does not admit a strictly coarser Hausdorff group topology. 
 We provide a sufficient and necessary condition for the minimality of the semidirect product $G\leftthreetimes P,$ where $G$ is a compact topological group and $P$ is a topological subgroup of $\Aut(G)$. 
    	We prove that $G\leftthreetimes P$ 
    	is minimal for every closed subgroup $P$ of $\Aut(G)$. 
    	In case $G$ is abelian, the same is true for every subgroup $P \subseteq \Aut(G)$.
 We show, in contrast, that there exist a compact two-step nilpotent group $G$ and a subgroup $P$ of $\Aut(G)$ such that $G\leftthreetimes P$ is not minimal. 
 This answers a question of Dikranjan.
    Some of our results were inspired by a work of Gamarnik \cite{Gam}.
\end{abstract}

\maketitle

\section{Introduction}
A Hausdorff topological group $G$ is \textit{minimal} (\cite{Doitch}, \cite{Steph}) if it does not admit a strictly coarser Hausdorff group topology or, equivalently, if every injective continuous group homomorphism $G \to P$ into a Hausdorff topological group is a topological group embedding.
For information on minimal groups we refer to the surveys \cite{CHR}, \cite{DikrSurv},
\cite{MegDik} and the book \cite{DPS}.

In \cite{MP1}
the two first-named authors study the minimality of the group $\H_+(X)$, where $X$ is a compact linearly ordered space and $\H_+(X)$ is the topological group of all order-preserving homeomorphisms
of $X$. In general, $\H_+(X)$ need not be minimal.
The first result in the present paper is Theorem \ref{not_pi_uniform_LOTS}, which shows that for a compact (partially) ordered spaces $X$ the compact-open topology on $\H_+(X,\leq)$ is minimal within the class of $\pi$-\emph{uniform topologies} (in the sense of \cite{Meg}). This result was inspired by results of Gamarnik \cite{Gam} and of the
first-named author \cite{Meg}. Following Nachbin \cite{Nachbin}, by a \emph{partially ordered topological space} we mean a topological space $X$ equipped with a partial order which is closed in $X \times X$.

Let $G$ be a compact topological group and denote  by $\Aut(G)$ its group of automorphisms. With the compact-open topology, $\Aut(G)$ becomes a topological group (which is not necessarily minimal). We denote its compact subgroup of inner automorphisms by $\Inn(G).$ 
The center of $G$ is denoted by $Z(G).$

One of the main objectives of this paper is to prove that
the semidirect product $G\leftthreetimes P$ is minimal for every closed subgroup $P$ of $\Aut(G)$ (Theorem \ref{t:Aut}). Using this result, as well as the Minimality Criterion,  the minimal groups $G\leftthreetimes P$ are fully characterized (Theorem \ref{thm:sol}).
This characterization shows, in particular, that if $G$ is abelian then $G\leftthreetimes P$ is minimal for every (not necessarily closed)
subgroup $P$ of $\Aut(G)$ (Corollary \ref{cor:min}). Furthermore, when $G$ is not abelian, 
the condition that $P$ is a closed subgroup of $\Aut(G)$ is essential to ensure the minimality of $G\leftthreetimes P$. Indeed,
negatively answering a question of Dikranjan,  
 in Example \ref{nil} we show that there exist a compact two-step nilpotent group $G$ and a subgroup $P$ of $\Aut(G)$ such that $G\leftthreetimes P$ is not minimal.

Note also that the compactness of $G$ cannot be replaced by precompactness. Indeed, 
there exist a minimal
precompact group $G$ and a two-element subgroup $P$ of $\Aut(G)$ such that  $G \leftthreetimes P$  is not minimal. See Eberhardt-Dierolf-Schwanengel \cite[Example 10]{Eberhardt} and also \cite[Example 4.6]{MegDik}. 
The latter example also demonstrates that,
in general, for two arbitrary minimal groups $G$ and $H $ the topological semidirect
product $G\leftthreetimes H$ may fail to be minimal. However, adding the requirement of completeness of $G$, one has the following:

\begin{lemma} \label{survey4.3} \textnormal{\cite[Theorem $4.3$]{MegDik}}
    If $G$ is complete (with respect to its two-sided uniformity), then the semidirect product $G\leftthreetimes H$ is minimal for minimal groups $G$ and $H$.
\end{lemma}

\begin{remark} \label{r:AutNotMin}
    Let $G$ be a complete minimal topological group and assume that $\Aut(G)$ is minimal. Then $G \leftthreetimes \Aut(G)$ is minimal by Lemma \ref{survey4.3}.
    So, in view of Theorem \ref{t:Aut}, it is important to note that
    there are compact groups $G$ such that $\Aut(G)$ is not minimal. Indeed, as it was explained in \cite[Section 5]{MegDik}, one may take   $G=(\Q,\textup{discrete})^*$, that is the Pontryagin dual of the discrete group $\Q$ of the rational numbers.
\end{remark}
For additional results concerning the minimality of $\Aut(G)$ see \cite{MegDik}.
For the minimality of semidirect products see for example \cite[Section 7.2]{DPS} and \cite[Section 4]{MegDik}.
More results about minimality of the homeomorphism groups 
can be found in \cite{Gartside,vanMill2012,MegDik,ChangGartside}.

\section{Preliminaries}
All topological spaces are assumed to be Hausdorff and completely regular, unless stated otherwise. 
The closure of a subset $A$ in a topological space will be denoted by $\overline{A}$.  
In what follows, every compact topological space will be considered as a uniform space with respect to its natural (unique) uniformity.
For a topological space $X$ we denote by $\H(X)$ its group of homeomorphisms and, if $X$ is ordered, $\H_+(X)$ denotes the group of all order-preserving homeomorphisms of $X$. With the compact-open topology $\tau_{co}$, $\H(X)$ becomes a topological group for every compact space $X$.

For a topological group $(G,\gamma)$ and its subgroup $H$ denote by $\gamma / H$ the natural quotient topology on the coset space $G/H$.

The following useful lemma can be found, for example, in \cite[Lemma $7.2.3$]{DPS}.
\begin{lemma}  \textnormal{(Merson's Lemma)} \label{merson}
    Let $(G, \gamma)$ be a not necessarily Hausdorff topological group and $H$ be a not necessarily closed subgroup of $G$. If $\gamma_1 \seq \gamma$ is a coarser group topology on $G$ such that $\gamma_1|_H=\gamma|_H$ and $\gamma_1/H=\gamma/H$, then $\gamma_1=\gamma$.
\end{lemma}

\subsection{$\pi$-uniform actions}

Let $\pi \co G \times X \to X$ be an action of a topological group $G$ on a topological space $X$.
Define two maps:
\begin{enumerate}
    \item \textit{$g$-translation}: $\pi ^g \co X \to X$, $\pi ^g (x)=gx$;
    \item \textit{$x$-orbit}: $\pi_x \co G \to X$, $\pi_x (g)=gx$.
\end{enumerate}

For a topological group $(G,\tau)$ we denote by $e_G$ (or simply $e$ when $G$ is understood) the identity element of $G$, and by $N_{g_0}(\tau)$ we denote the local base of $G$ at $g_0\in G$.

\begin{defi} \label{d:pi} \cite{Me-t,Meg}
    Let $\pi \co G \times X \to X$ be an action of a topological group $(G,\tau)$ on a Hausdorff uniform space $(X,\mathcal U)$.
    The uniformity (or, the action) is said to be:
    \begin{enumerate} \label{bounded}
        \item \textit{$\pi$-uniform at $e$} or \textit{quasibounded} if for every $\varepsilon \in \mathcal U$ there exist $\delta \in \mathcal U$ and $U\in N_e(\tau)$ such that $(gx,gy) \in  \varepsilon$ for every $(x,y) \in \delta$ and $g \in U$.
        \item \textit{$\pi$-uniform} if for every $g_0 \in G$ and for every $\varepsilon \in \mathcal U$ there exist $\delta \in \mathcal U$ and $U\in N_{g_0}(\tau)$ such that $(gx,gy) \in  \varepsilon$ for every $(x,y) \in \delta$ and $g \in U$.
    \end{enumerate}
\end{defi}

The notion of a $\pi$-uniform action, defined in \cite{Me-t, Meg}, was originally used to study compactifications of $G$-spaces. Later it was employed by Gamarnik \cite{Gam} to prove that for a compact space $X$, the compact-open topology on $\H(X)$ is minimal within the class of $\pi$-uniform topologies. More applications of $\pi$-uniformity can be found in \cite{Kozlov} and in \cite{Chatyrko}.

\begin{defi} \label{d:unif}
    Let $X$ be a compact space and $G$ a subgroup of $\H(X)$. A Hausdorff group topology $\tau$ on $G$ is said to be \textit{$\pi$-uniform} if the natural action $(G,\tau) \times X \to X$ is $\pi$-uniform with respect to the unique compatible uniformity on $X$.
\end{defi}

For a topological group $X$ denote by $\mathcal U_l,\mathcal U_r, \mathcal U_{l\vee r}$ the left, right and two-sided uniform structures on $X$, respectively.
We give here some simple but useful facts for further use.

\begin{lemma} \label{prtq}

    Let $G$ be a topological group and $X$ is a uniform space.
    \begin{enumerate}
        \item \label{contOrbMaps}
        If  $\pi \co G \times X \to X$ is a $\pi$-uniform action and all orbit maps
         ${\pi_x \co G \to X}$ are continuous, then $\pi$ is continuous.

        \item \textnormal{\cite[Theorem 1.2]{me-sing}} If $X$ is a topological group and  ${\pi\co G \times X \to X}$ is an action by continuous automorphisms, then the action is $\pi$-uniform with respect to $\mathcal U \in \{\mathcal U_l,\mathcal U_r, \mathcal U_{l\vee r} \}$ if and only if $\pi$ is continuous at $(e_G,e_X)$.
    \end{enumerate}
\end{lemma}

\begin{lemma} \label{l:co} \textnormal{\cite[Theorem $2$]{Arens45}}
    Let $X$ be a compact space and let $G$ be a subgroup of  $\H(X)$. If $\tau$ is a  group topology on $G$ such that the action $G \times X \to X$ is continuous, then $\copt \subseteq \tau$.
\end{lemma}

\subsection{Ordered Spaces}

By \textit{order} we mean a reflexive, antisymmetric and transitive relation.

\begin{defi} \textnormal{(Nachbin \cite{Nachbin})}
    A \textit{topological ordered space} is a triple $(X,\leq, \tau)$ where $X$ is a set, $\leq$ is an order on $X$, $\tau$ is a topology on $X$  and the graph of the order  $Gr(\leq)=\{(x,y) : x\leq y \}$ is closed in $X\times X$.
    In particular, a \textit{compact topological ordered space} is a topological ordered space that is compact. Since in this paper all ordered spaces are topological, we will sometimes omit the term "topological".
\end{defi}

\begin{remark} \label{r:part}
    Every Hausdorff topological space $X$ is a topological ordered space
    with respect to the trivial order (equality). Indeed, the diagonal is closed in $X \times X$  exactly when $X$ is Hausdorff.
\end{remark}

A subset $Y \subseteq X$ is said to be \textit{decreasing} if $x \leq y\in Y$ implies $x \in Y$. Similarly one defines an \textit{increasing} subset.

\begin{lemma} \label{disjNbhd} \textnormal{\cite[Prop. 1]{Nachbin}}
    Let $(X, \leq)$ be an ordered set and let $\tau$ be a topology on $X$. The following conditions are equivalent:
    \begin{enumerate}
        \item $(X,\leq, \tau)$ is a topological ordered space (that is, $Gr(\leq)$ is $\tau$-closed in $X \times X$);
        \item if $x \leq y$ is false, then there exist: an increasing neighborhood $W$ of $x$ and a decreasing neighborhood $V$ of $y$ such that $V \cap W = \emptyset$.
    \end{enumerate}
\end{lemma}

\begin{lemma} \label{l:closed}
    Let $(X,\leq, \tau)$ be a compact partially ordered space.
    Denote by $\H_+(X)$ the group of all order-preserving homeomorphisms of $X$. Then $\H_+(X)$ is a closed subgroup of the topological group $\H(X)$.
\end{lemma}
\begin{proof} Since $Gr(\leq)$ is $\tau$-closed in $X \times X$, the subgroup $\H_+(X)$ is even pointwise closed in $\H(X)$.
\end{proof}

\subsection{Limit points and ultrafilters}

All definitions and results of this subsection can be found, for example, in \cite[Chapter $1$, Section $7$]{Bourbaki}.
Let $X$ be a topological space and $\mathcal{J}$ is a filter on $X$. A point $x\in X$ is said to be a \textit{limit point of a filter} $\mathcal{J}$, if $\mathcal{J}$ is finer than the neighborhood filter $\mathcal{N}_x$ of $x$. We also say that $\mathcal{J}$ is convergent to $x$. A point $x$ is called a \textit{limit point of a filter base} $\mathcal{B}$ on $X$, if the filter whose base is $\mathcal{B}$ converges to $x$.
Let $f$ be a mapping from a set $X$ to a topological space $Y$, and let $\mathcal{J}$ be a filter on $X$. A point $y \in Y$ is a \textit{limit point of $f$ with respect to the filter $\mathcal{J}$} if $y$ is a limit point of the filter base $f(\mathcal{J})$.

\begin{prop}\textnormal{\cite{Bourbaki}}
    \begin{enumerate}
        \item
        If $\mathcal{B}$ is an ultrafilter base on a set $X$ and if $f$ is a mapping from $X$ to $Y$, then $f(\mathcal{B})$ is an ultrafilter base on $Y$.
        \item  Let $f$ be a mapping from a set $X$ into a topological space $Y$, and let $\mathcal{J}$ be a filter on $X$. A point $y\in Y$ is a limit point of $f$ with respect to the filter $\mathcal{J}$ if and only if  $f^{-1}(V) \in \mathcal{J}$ for each neighborhood $V$ of $y$ in $Y$.
        \item  If $X$ is a compact Hausdorff space, then every ultrafilter on $X$ converges to a unique point.
    \end{enumerate}
\end{prop}

We can sum these propositions as follows:
\begin{corol} \label{limitpoint}
    Let $\mathcal{J}$ be an ultrafilter on a set $E$ and let $f$ be a mapping from $E$ to a compact space $X$. Then there exists a unique point $\widetilde{x} \in X$ such that each neighborhood $O$ of $\widetilde{x}$ satisfies $f^{-1}(O) \in \mathcal{J}$. That is, $\widetilde{x}$ is the limit point of $f$ with respect to $\mathcal{J}$.
\end{corol}

\section{$\pi$-uniform topologies on $\H_+$ and $\Aut_+$}

The following theorem is an extended version of a result of Gamarnik \cite[Prop. 2.1]{Gam}.

\begin{thm} \label{not_pi_uniform_LOTS}
    Let $(X,\tau_{\leq})$ be a compact partially ordered space
    and 
    let $P$ be a closed subgroup of $\H_+(X),$  the group of all order-preserving homeomorphisms of $X$. Then the compact-open topology $\tau_{co}$ is minimal
    within the class of $\pi$-uniform topologies on $P.$
\end{thm}

\begin{proof}
    Assuming the contrary, suppose that there exists a
    $\pi$-uniform group topology $\tau$ on $P$ such that
    $ \tau \varsubsetneqq\copt$. Let $\pi \co P\times X \to X$ be the natural action of  $P$ on $X$. If all orbit maps are continuous, then, by Lemma \ref{prtq}.1, $\pi$ is continuous and, by Lemma \ref{l:co}, $\copt \seq \tau$.
    So we can assume that there exists an orbit map that is not continuous (at the identity). That is, there exists $x_0 \in X$ such that $\pi_{x_0}\co P \to X$ is not continuous at $e \in P$. Thus, denoting by $\mathcal U$ the natural uniformity on $X$, there exists $\eps_0 \in \mathcal U$ such that for all $U \in N_e(\tau)$ there exists $g_U \in U$ for which
    \begin{equation} \label{gUU}
        (g_Ux_0, x_0) \notin \eps_0.
    \end{equation}

    For a given $U \in N_e (\tau)$ define $F(U)=\{V \in N_e(\tau) : V \seq U \}$.
    Denote by $\mathcal F$ the filter on the set $N_e(\tau)$ generated by the filter base $\{F(U)\} _{U\in N_e(\tau)}$. Since every filter is contained in an ultrafilter, choose an ultrafilter $\mathcal{J}$ on $ N_e (\tau)$ that contains $\mathcal F$.

    For each $x \in X$ define a map $f_x \co N_e(\tau) \to X$  by $f_x(U)= g_Ux$ for $g_U$ that satisfies (\ref{gUU}). Let $\widetilde{x}$ be the limit point of $f_x$ with respect to the ultrafilter $\mathcal{J}$ given by Corollary \ref{limitpoint}. Define the following transformation

    \begin{equation} \label{the_map_h}
        h \co X \to X , \ \ h(x)=\widetilde{x}.
    \end{equation}

    In the rest of the proof we show that $h$ is a nontrivial order-preserving homeomorphism that belongs to every neighborhood of the identity element in $(P,\tau)$, in contradiction to $\tau$ being a Hausdorff group topology.

    \begin{claim} \label{h}
        The map $h$ defined by \textup{(\ref{the_map_h})} is a nontrivial homeomorphism in $P$.
    \end{claim}
    \begin{proof}
        We break the proof into five steps. \vspace{0.2cm}

        Step $\bf{1}$. In order to prove that $h$ is one-to-one, assume for a  contradiction that there exist $x,y,z \in X$ such that $h(x)=h(y)=z$ and $x\neq y$. Choose an entourage  $\eps \in \mathcal U$ such that $(x,y) \notin \eps$. The action is $\pi$-uniform at the identity, and thus there exist $U_\eps \in N_e(\tau)$ and $\delta_\eps \in \mathcal U$ such that $(gx,gy) \in \eps$ for every $(x,y) \in \delta_\eps$ and $g\in U_\eps$. Choose a symmetric $\delta \in \mathcal U$ satisfying $\delta^2 \seq \delta_\eps$.

        By assumption $z=h(x)$ is the limit point of $f_x$ with respect to $\mathcal{J}$. That is, for every entourage in the uniformity, and in particular for $\delta$, we have:
        $$A(x,\delta)=\{U \in N_e(\tau) : \ (g_Ux,z) \in \delta \} \in \mathcal{J}.$$ Similarly,
        $$A(y,\delta)=\{U \in N_e(\tau) : \ (g_Uy,z) \in \delta \} \in \mathcal{J}.$$ Also, since $F(U_\eps^{-1}) \in \mathcal{J}$, the intersection $A(x,\delta) \cap A(y,\delta) \cap F(U_\eps^{-1})$ is not empty. If $U_0 \in A(x,\delta) \cap A(y,\delta) \cap F(U_\eps^{-1})$ and $g_{U_0} \in U_0$, then $g_{U_0} \in U_\eps^{-1}$ (and thus $g_{U_0}^{-1} \in U_\eps$), $(g_{U_0}x,z) \in \delta$, $(g_{U_0}y,z) \in \delta$ (and thus $(g_{U_0}x,g_{U_0}y) \in \delta^2 \seq \delta_\eps$). By the choice of $\delta_{\eps}$ and $U_{\eps}$ we have
        $(g_{U_0}^{-1}g_{U_0}x,g_{U_0}^{-1}g_{U_0}y)=(x,y) \in \eps$, and this contradicts the choice of $\eps$. Therefore $h$ is one-to-one.

        \vspace{0.2cm}

        Step $\bf 2$. To prove that $h$ is onto, for a given $y \in X$ we will find $x \in X$ such that $h(x)=y$. Fix $y\in X$ and consider the map $N_e(\tau) \to X$ given by $U \mapsto g_U^{-1}y$. Let $x$ be the limit point of this map with respect to the ultrafilter $\mathcal{J}$. To show that $h(x)=y$ we will show that $y$ is the limit point of $f_x$ with respect to $\mathcal{J}$. Let $\eps \in \mathcal U$ be an arbitrary entourage and choose $U_\eps$, $\delta_\eps$ from the definition of $\pi$-uniform topology. Since $x$ is defined as the limit point of $g_U^{-1}y$, we know that $$B(y,\delta_{\eps})=\{U \in N_e(\tau) : \ (g_U^{-1}y,x) \in \delta_\eps \} \in \mathcal{J},$$ and since $F(U_\eps)$ is also an element of $\mathcal{J}$,  the intersection $B(y,\delta_{\eps}) \cap F(U_\eps)$ is not empty. Let $U \in B(y,\delta_{\eps}) \cap F(U_\eps)$. Then for $g_U \in U \seq U_\eps$ and $(g_U^{-1}y,x)\in \delta_\eps$ we have $(y,g_Ux)\in \eps$ (by the choice of $U_\eps,\delta_\eps$). This last condition is satisfied by all $U\in B(y,\delta_{\eps}) \cap F(U_\eps)$ and, therefore, $$\{U\in N_e(\tau)  :  (y,g_Ux) \in \eps \} \in \mathcal{J}$$ (since $B(y,\delta_{\eps}) \cap F(U_\eps) \in \mathcal{J}$ and $B(y,\delta_{\eps}) \cap F(U_\eps) \seq \{U\in N_e(\tau)  :  (y,g_Ux) \in \eps \}$.) This holds for all $\eps \in \mathcal U$, which proves that $y$ is the limit point of $f_x = g_Ux$ with respect to $\mathcal{J}$. And that, in turn, proves that $h(x)=y$ and therefore $h$ is onto. \vspace{0.3cm}

        \vspace{0.2cm}

        Step $\bf 3$. In order to prove that $h$ is (uniformly) continuous we will show that for every $\varepsilon \in \mathcal U$ there exists $\delta \in \mathcal U$ such that $(h(x),h(y))\in \varepsilon$ for all $(x,y)\in \delta$.  Let $\eps_0 \in \mathcal U$ and choose
        a symmetric entourage $\eps \in \mathcal U$ such that $\eps^3 \seq \eps_0$. Choose $\delta_\eps,U_\eps$ from Definition \ref{d:pi} of  $\pi$-uniformity. We will show that if $(x,y) \in \delta_\eps$ then $(h(x),h(y)) \in \eps_0$. Let $(x,y) \in \delta_\eps$ and assume for a contradiction that $(h(x),h(y)) \notin \eps_0$. This means that if $t_1,t_2$ satisfy $(h(x),t_1)\in \eps \ , (h(y),t_2) \in \eps$, then (since $\eps^3 \seq \eps_0$) we have
        \begin{equation} \label{t''}
            (t_1,t_2) \notin \eps.
        \end{equation}
        Since $h(x)$ is the limit point of $g_Ux$, $A(x,\eps)=\{U \in N_e(\tau) : (g_Ux,h(x)) \in \eps \} \in \mathcal{J}$ and, similarly, $A(y,\eps)=\{U \in N_e(\tau) : (g_Uy,h(y)) \in \eps \} \in \mathcal{J}$. Also, since $F(U_\eps) \in \mathcal{J}$, the intersection $A(x,\eps) \cap A(y,\eps) \cap F(U_\eps)$ is not empty. Let $V \in A(x,\eps) \cap A(y,\eps) \cap F(U_\eps)$. In particular $(g_Vx,h(x)) \in \eps$ and $(g_Vy,h(y)) \in \eps$. Next, from (\ref{t''}) it follows that $(g_Vx,g_Vy) \notin \eps$. But $V\seq U_\eps$ and thus $g_V \in U_\eps$. Since $(x,y) \in \delta_\eps$ we know, by the definition of $\pi$-uniformity, that $(g_Vx,g_Vy)\in \varepsilon$ and this is the desired contradiction.
        \vspace{0.2cm}

        Step $\bf 4$. To see that $h$ is not trivial, recall that from (\ref{gUU}) we have $x_0 \in X$ and $\eps_0 \in \mathcal U$ such that $(g_Ux_0,x_0) \notin \eps_0$ for every $U \in N_e(\tau)$. This implies that $h(x_0) \neq x_0 $.

        \vspace{0.2cm}

        Step $\bf 5$.
        Finally, we show that  $h\in P.$ Denote by $A$ the set of all $g_U$ that satisfy  (\ref{gUU}), that is $A=\{g_U:U\in N_e(\tau)\}$. Since $P$ is closed, we have $A\seq \overline{A} \seq P$, where $\overline{A}$ is the closure of $A$ in $\H_+(X)$ with respect to the compact-open topology $\tau_{co}$.

        We are going to show that $h\in \overline{A}$. That is, we need to show that every neighborhood of $h$ contains some $g_U$. 	Since $X$ is compact,
        $\tau_{co}$ coincides with the topology of uniform convergence, and hence for $\varepsilon \in \mathcal U$ a basic neighborhood of $h$ is of the form $\widetilde{\varepsilon}(h)=\{f\in P:(f(x),h(x)) \in \varepsilon \ \forall x\in X\}$. Therefore, for every $\eps \in \mathcal U$ we will find $U\in N_e(\tau)$ such that $\forall x\in X : (g_Ux,hx)\in \eps$.

        Fix an $\eps\in \mathcal U$ and choose
        a symmetric entourage
        $\delta\in \mathcal U$ such that $\delta^3\seq \eps$.
        Since $\tau$ is $\pi$-uniform at the identity, for $\delta$ there exist $\eta\in \mathcal U$ and $U_0\in N_e(\tau)$ such that $$(\forall (x,y)\in \eta)(\forall g\in U_0):(gx,gy)\in \delta.$$
        Since $g_U \in U$ we obtain, in particular,
        \begin{equation}\label{gcp1}
            \forall (x,y)\in \eta \ \forall U\subseteq U_0: (g_{U}x,g_{U}y)\in \delta.
        \end{equation}
        Since $h$ is uniformly continuous, for $\delta$ there exists $\kappa\in \mathcal U$ such that
        \begin{equation}\label{gcp2}
            \forall (x,y)\in \kappa: (hx,hy)\in \delta.
        \end{equation}
        If necessary, we intersect $\kappa$ with $\eta$ to ensure that $\kappa\seq \eta$, which we will need later in the proof.

        Now, since $X$ is compact, for $\kappa$ that satisfies (\ref{gcp2}) there exists a finite collection of points $x_1,\dots,x_n \in X$ such that $\kappa[x_1]\cup\dots\cup\kappa[x_n]=X$.

        We will show that	there exists $U_1\seq U_0$ such that for all $i\in \{1,\dots,n\}$
        \begin{equation}\label{gcp3}
            (g_{U_1}x_i,hx_i)\in \delta.
        \end{equation}
        For a fixed index  $i\in \{1,\dots,n\}$, since $hx_i=\widetilde{x_i}$ is the limit of $f_{x_i}$ with respect to $\mathcal{J}$, we have $A(x_i,\delta)=\{U\in N_e(\tau): (g_Ux_i,hx_i)\in \delta   \}\in \mathcal{J}$. Recall that $F(U_0)\in \mathcal{J}$ and thus the intersection $F(U_0)\bigcap (\bigcap_{i=1}^n A(x_i,\delta))$ is not empty. Choose a set $U_1$ from this intersection. Then $U_1\seq U_0$ and for all  ${i\in \{1,\dots,n\}}$ we have $(g_{U_1}x_i,hx_i)\in \delta$, as required.
        We claim that $U_1$ is the desired neighborhood. That is, for every $x\in X$ we have $(g_{U_1}x,hx)\in \eps$.
        Indeed, fix an $x\in X$. There exists  $i\in \{1,\dots,n\}$ such that $x\in \kappa [x_i]$. At this point recall that $\kappa \seq \eta$ and $g_{U_1}\in U_1\seq U_0$. Since $(x,x_i)\in \kappa\seq \eta$, by (\ref{gcp1}) we have
        \begin{equation}\label{gcp4}
            (g_{U_1}x,g_{U_1}x_i)\in \delta.
        \end{equation}
        Also,  by (\ref{gcp2}) we have
        \begin{equation}\label{gcp5}
            (hx,hx_i)\in \delta.
        \end{equation}

        Now, combining (\ref{gcp4}), (\ref{gcp3}) and (\ref{gcp5}) we get
        $$(g_{U_1}x,g_{U_1}x_i), (g_{U_1}x_i,hx_i), (hx_i,hx) \in \delta^3.$$

        Therefore, for every $x\in X$ we have $(g_{U_1}x,hx)\in \delta^3\seq \eps$ and thus $h\in \overline{A}$, as required.
    \end{proof}

    The following claim shows that $\tau$ is not Hausdorff.
    \begin{claim} \label{notHaus}
        For every $U \in N_e(\tau)$, $h \in U$.
    \end{claim}
    \begin{proof}
        For  $g \in P$ and  $\eps \in \mathcal U$ define $$\tilde{\eps}(g) =\{f \in P: (g(x),f(x)) \in \eps  \text{ for all } x \in X \} .$$
        It can be easily verified that $\{\tilde{\eps}(g) \}_{\eps \in \mathcal U} $  is a local base of neighborhoods for every point $g \in P$, with respect to the compact-open topology on $P$. In order to prove the statement, it suffices to show that $h \in [\tilde{\eps}^3(e)]^{-1}U_0$ for every $U_0 \in N_e(\tau)$ and for every $\eps \in \mathcal U$. Indeed, for each $U \in N_e(\tau)$ we can find $U_0 \in N_e(\tau)$ such that $U_0^2 \seq U$ and $U_0^{-1}=U_0$. But $\tau \seq \copt$, and $\{\tilde{\eps}(e)   \}_{ \eps \in \mathcal U}$ is a local base at $e$, thus there exists $\eps \in \mathcal U$ with $\tilde{\eps}^3(e) \seq U_0$.  Therefore $[{\tilde{\eps}}^3(e)]^{-1}U_0 \seq U_0^{-1}U_0 \seq U$.

        Let $\eps \in \mathcal U$ and $U_0\in N_e(\tau)$. Choose $\delta_\eps \in \mathcal U$ and $U_\eps \in N_e(\tau)$ from the definition of $\pi$-uniform topology. For $x \in X$ define $A(h^{-1}(x),\eps)=\{U \in N_e(\tau) : (g_Uh^{-1}(x),x)\in \eps  \}$. Since $h(h^{-1}(x))=x$, from the definition of $h$ we have $A(h^{-1}(x),\eps) \in \mathcal{J}$. Indeed, $x=h(h^{-1}(x))=\widetilde{h^{-1}(x)} $, $x$ is the limit point of the map $f_{h^{-1}(x)}\co N_e(\tau)\to X$ defined by $f_{h^{-1}(x)}(U)=g_Uh^{-1}(x)$. Since $h$ (and thus $h^{-1}$) is uniformly continuous, we can choose $\alpha \in \mathcal U$ such that $\alpha \seq \eps$ and
        \begin{equation} \label{implication}
            (t_1,t_2) \in \alpha \Rightarrow (h^{-1}(t_1),h^{-1}(t_2)) \in \delta_\eps.
        \end{equation}

        Since $X$ is compact, we can find a finite subset $\{x_1,x_2,...,x_n  \} \seq X$ such that for every $x\in X$ there exists $1 \leq i \leq n$ for which  $(x,x_i) \in \alpha$.
        Let
        $$
        U\in \left(\bigcap_{i=1}^n A(h^{-1}(x_i),\eps)\right)\bigcap F\left(U_\eps \cap U_0\right).
        $$
        For every $x\in X$ there exists $i$ such that $(x,x_i) \in \alpha$ and from (\ref{implication}) we have $(h^{-1}(x),h^{-1}(x_i)) \in \delta_\eps$. Since $U \seq U_\eps$, by the choice of $U_\eps$ and $\delta_\eps$ we have $(g_Uh^{-1}(x),g_Uh^{-1}(x_i)) \in \varepsilon$.
        Since $U \in A(h^{-1}(x_i),\eps)$, it follows that $(g_Uh^{-1}(x_i),x_i) \in \eps $. Recalling that $(x,x_i) \in \alpha \seq \eps$ we obtain $(g_Uh^{-1}(x),x) \in \eps^3$, and therefore $g_Uh^{-1} \in \tilde{\eps}^3(e)$. But since $g_U \in U \seq U_0$, we get $h\in [\tilde{\eps}^3(e)]^{-1}U_0$, and this completes the proof.
    \end{proof}
    Claims \ref{h} and \ref{notHaus} complete the proof of Theorem \ref{not_pi_uniform_LOTS}.
\end{proof}

By Remark \ref{r:part}, every Hausdorff topological space can be viewed as an ordered topological space. Therefore, Theorem \ref{not_pi_uniform_LOTS}  directly yields the following:

\begin{thm}  \label{t:usualcase}
    Let $(X,\tau)$ be a compact topological space
    and 
    let $P$ be a closed subgroup of $\H(X),$  the group of all  homeomorphisms of $X$. Then the compact-open topology $\tau_{co}$ is minimal
    within the class of $\pi$-uniform topologies on $P.$
\end{thm}

By Lemma \ref{l:closed} $\H_+(X)$ is a closed subgroup of $\H(X)$ for every partially ordered compact space $X$.
So, if $P$ is a closed subgroup of $\H_+(X)$ it is also a closed subgroup of $\H(X).$ Therefore, by Theorem \ref{t:usualcase} the compact-open topology $\tau_{co}$ is minimal
within the class of $\pi$-uniform topologies on $P.$ It follows that Theorem \ref{not_pi_uniform_LOTS} can be derived back from Theorem \ref{t:usualcase}.

Let us extend Theorem \ref{not_pi_uniform_LOTS} in some algebraic setting.
Let $\omega\co K\times K\to K$ be a binary operation on a compact space $K$.
Denote by $\Aut(K)$ the group of all topological automorphisms of the  structure $(K,\omega)$. If $K$ is a compact ordered space, then we denote by $\Aut_+(K)$ the group of
all order preserving
automorphisms of $(K,\omega,\leq)$. Note that $\Aut_+(K)=\Aut(K) \cap \H_+(K)$.
Since  $\Aut_+(K)$ is a closed subgroup of $\H_+(K)$ by Theorem \ref{not_pi_uniform_LOTS} we obtain:

\begin{corol}\label{AutIsMinimal}
    If $K$ is a compact ordered space with a binary operation $\omega$ and $P$ is a closed subgroup of $\Aut_+(K),$ then   the compact-open topology is minimal within the class of $\pi$-uniform topologies   on $P$.
\end{corol}

By Remark \ref{r:part}, every topological group can be viewed as an ordered topological space equipped with a group operation. Therefore, by Corollary \ref{AutIsMinimal} we get the following:

\begin{corol} \label{c:m}
    If $K$ is a compact topological group and $P$ is a closed subgroup of $\Aut(K),$  then the compact-open topology is minimal within the class of $\pi$-uniform topologies on $P$.
\end{corol}

\section{Minimality of $G \leftthreetimes P$ where  $P$ is closed}
The main goal of this section is to prove that for every compact topological group $G$, the natural semidirect product $G \leftthreetimes P$ is a minimal topological group
for every closed subgroup $P \leq  \Aut(G)$ (Theorem \ref{t:Aut}).

We need the following technical result which is inspired by \cite[Prop. 2.6]{Meg95} (see also \cite[Theorem 4.13]{MegDik}).

\begin{thm} \label{contati}
    Let $(M,\gamma)$ be a topological group, $X$ and $G$ are subgroups of $M$ such that $M$ is \emph{algebraically} a semidirect product $M=X \leftthreetimes_{\alpha} G$.
    Assume that the topological subgroup $(X,\gamma|_X)$ of $(M,\gamma)$ is compact.
    Then the action $$\alpha\co(G, \gamma/X)\times (X,\gamma|_X) \to (X, \gamma|_X)$$ is continuous at $(e_G,e_X)$, where $\gamma/X$ is the coset topology on $G$ induced by $\gamma$.
\end{thm}

\begin{proof}
    Let $pr\co M \to G=M/X$, $(x,g) \mapsto g$, denote the canonical projection. Algebraically $M/X=\{X\times \{g\}\}_{g\in G}$, which allows us to identify $G$ with $M/X$, and thus the topological group $(G,\gamma/X)$ is well defined.

    To show that $\alpha$ is continuous at $(e_G,e_X)$ let
    $O \in \gamma|_X$ 	be a neighborhood of $e_X$. We will find neighborhoods $P$ of $e_G$
    in $(G, \gamma/X)$ and $U$ of $e_X$ in $(X, \gamma|_X)$ such that $\alpha(P \times U) \seq O$.

    Since $X$ is a compact group, there exists a neighborhood $O_1$ of $e_X$ such that for all $x\in X$ we have $x^{-1}O_1x\seq O$.
    The restriction $M\times X \to X$, $(a,x) \mapsto axa^{-1}$ of the conjugation $M \times M \to M$ in the topological group $(M,\gamma)$ is (well-defined, because $X$ is a normal subgroup of $M$) continuous at $(e_M,e_X)$. Therefore, for $O_1$ there exist a neighborhood $U$ of $e_X$ in $(X, \gamma|_X)$) and a neighborhood $V$ of $e_M$ in $(M,\gamma)$ such that $vUv^{-1} \seq O_1$ for all $v\in V$.

    Consider the canonical projection $pr\co M \to G=M/X$. Then $P:=pr(V) \in \gamma/X$ is a neighborhood of $e_G$ in $(G, \gamma/X)$.
    We claim that $P$ and $U$ satisfy the needed conditions above. That is, we want to show that $\alpha (g,z):=g(z) \in O$ for all $(g,z) \in pr(V)\times U$. Indeed, if $g \in pr(V)$ there exists $x \in X$ such that $(x,g) \in V$, and recall that $z \in U$ is in fact $(z,e_G)$. We know that $vzv^{-1} \in O_1$. Therefore, $$vzv^{-1}=(x,g)(z,e_G)(x,g)^{-1}= (x g(z),g)(g^{-1}(x^{-1}),g^{-1})=$$
    $$=(x g(z) g(g^{-1}(x^{-1})),e_G)=
    (x g(z)x^{-1},e_G)=x g(z)x^{-1} \in O_1.$$ Thus $\alpha (g,z)\in x^{-1}O_1x \seq O$, which completes the proof.
\end{proof}

\begin{thm} \label{t:Aut}
    If $G$ is a compact topological group, then $G \leftthreetimes P$ is minimal for every closed subgroup $P$ of $\Aut(G)$.
\end{thm}
\begin{proof}
    Let $\tau$ be the given topology on $G$, and $\copt$ the compact-open topology on $P \subseteq \Aut(G)$. Denote by $\gamma$ the product topology on $G \leftthreetimes P$, and by $e=id_G \in P$ the identity automorphism. Assume that $\gamma_1 \seq \gamma$ is a coarser Hausdorff group topology on $G \leftthreetimes P$.	
    Since $G$ is compact we have $\gamma_1|_G=\gamma|_G=\tau$.

    The action $$\alpha\co (P, \gamma_1/G)\times (G,\gamma_1|_G) \to (G,\gamma_1|_G)$$ is continuous at the identity $(id_G,e_G)$ by  Theorem \ref{contati}.
    Furthermore, $\gamma_1/G$ is a Hausdorff topology on $\Aut(G)$  since $G$ is a compact (hence, closed) subgroup of the Hausdorff group $(G \leftthreetimes P, \gamma_1)$. Therefore $\gamma_1/G$ is an $\alpha$-uniform topology on $P$ (Lemma \ref{prtq}.2
    and Definition \ref{d:unif}).

    Since $\gamma_1/G \seq \gamma/G = \copt$ and $\copt$ is minimal within the class of $\alpha$-uniform topologies on $P$
    (Corollary \ref{c:m}),
    we have  $\gamma_1/G = \gamma/G $. Finally, using Merson's Lemma \ref{merson}, we deduce that $\gamma_1 =\gamma$ and that concludes the proof.

\end{proof}

\section{Minimality of $G \leftthreetimes P$}

	In this section we provide an equivalent condition 
	 (see Theorem \ref{thm:sol}) for the non-minimality of $G \leftthreetimes P,$  where $G$ is a compact group and $P\leq \Aut(G)$. 
	As a corollary we obtain that if the compact group $G$ is  \textit{abelian}, then Theorem \ref{t:Aut} holds for all (thus not necessarily closed) subgroups $P$ of $\Aut(G)$. 
	Theorem \ref{thm:sol} also allows us to construct relevant counterexamples (Example \ref{nil} and Theorem \ref{thm:bool}).  
	
	We use 
	the well known Minimality Criterion which, for compact groups, can be traced back to 
	Stephenson \cite{Steph}  and Prodanov \cite{P1}. 
	 Note that Banaschewski \cite{Ban} generalized this criterion  by proving it for minimal topological algebras.
		
		First recall that
		a subgroup $H\leq G$ of a topological group $G$ is said to be {\it essential} in $G$
		if $H\cap N$ is nontrivial for every nontrivial  closed normal subgroup $N$ of $G.$
		
		The following theorem can be found, for example, in \cite[Theorem 2.5.1]{DPS}.
		\begin{thm}
			[Minimality Criterion]Let $G$ be a topological group and $H$ its dense subgroup. $H$ is minimal if and only if $G$ is minimal and $H$ is essential in $G.$
		\end{thm}
	\begin{thm}\label{thm:sol} Let $G$ be a compact group and $P\leq \Aut(G).$   
		Then	the following two conditions are equivalent:  
			\begin{enumerate}
				\item $G\leftthreetimes P$ is not minimal.
				\item There exists a closed nontrivial subgroup $H$ of $G$ satisfying the following conditions: \begin{enumerate}[(a)] 
					\item $H\cap Z(G)=\{e_G\}.$
					\item $H$ is $P$-invariant (that is, $f(H)\subseteq H$ for every $f\in P$).
					\item  
					$\Gamma(H) \subseteq \overline{P}$
					and $\Gamma(H)\cap P=\{e_P\},$ where $\Gamma \co G\to \Inn(G)$ is the 
					natural homomorphism defined by $\Gamma(g)=\gamma_g.$
				\end{enumerate} \end{enumerate}\end{thm}
				\begin{proof}
					$(2)\Rightarrow (1)$: Let $H$ be a closed nontrivial subgroup  of $G$ satisfying  conditions $(a)-(c).$ By
					Theorem \ref{t:Aut}
					$G \leftthreetimes \overline{P}$ is minimal. Clearly,  $G \leftthreetimes P$ is a dense subgroup of $G \leftthreetimes \overline{P}.$ We are going to construct a nontrivial closed normal subgroup $N$ of $G \leftthreetimes \overline{P}$ such that 
					$N\cap (G \leftthreetimes P)$ is trivial. Using the Minimality Criterion this will imply that $G \leftthreetimes P$ is not minimal.
					Let $N:=\{(h,\gamma_h^{-1})| \ h\in H\}.$ Since $H$ is a 
					compact nontrivial subgroup of $G$ and $\Gamma(H)$ is a 
					compact 
					 subgroup of $\overline{P},$ we obtain that $N$ is a  nontrivial 
					 compact (hence closed) 
					 subgroup of $G \leftthreetimes \overline{P}.$  Being $P$-invariant and closed the subgroup $H$ is also $\overline{P}$-invariant. This implies that $N$ is normal in $G \leftthreetimes \overline{P}.$ Indeed, $(g,f)(h,\gamma_h^{-1})(g,f)^{-1}=(f(h),\gamma_{f(h)}^{-1}).$ Let $(h,\gamma_h^{-1})$ be a nontrivial element of $N.$ Then we necessarily have $h\neq e_G.$  By 
					 condition $(a)$ this implies that $\gamma_h^{-1}$ is a nontrivial element of $\Gamma(H).$ It follows from 
					 $(c)$ that $\gamma_h^{-1}\notin P.$ Therefore, $(h,\gamma_h^{-1})\notin  G \leftthreetimes P$ and we conclude that 
					 $N\cap (G \leftthreetimes P)$ is trivial as needed.\\
					$(1)\Rightarrow (2):$ Assume that $G\leftthreetimes P$ is not minimal. If follows from  
					Theorem \ref{t:Aut} and the Minimality Criterion that $G \leftthreetimes P$ is not essential in $G\leftthreetimes \overline{P}.$ So, there exists a  nontrivial closed normal subgroup $N$ of $G\leftthreetimes \overline{P}$ such that  $N\cap (G\leftthreetimes P)$ is trivial. We will prove $(2)$ by showing that $N=\{(h,\gamma_h^{-1})| \ h\in H\},$ where   $H$ is a nontrivial closed subgroup of $G$ satisfying conditions $(a)-(c).$ First, let us show that every element in $N$ has the form $(h,\gamma_h^{-1}),$ for some $h\in G.$  Indeed, otherwise, there exists $u_1=(h,f)\in N$ such that $f\neq \gamma_h^{-1}.$ Choose $u_2=(g,e_P),$ where $f(g)\neq \gamma_{h}^{-1}(g).$ Since $N$ is normal 
					in $G\leftthreetimes \overline{P}$
					 the commutator
					 $$[u_1,u_2]=(h\cdot f(g)\cdot h^{-1}\cdot g^{-1},e_P)$$ is an element of $N.$ The inequality $f(g)\neq \gamma_{h}^{-1}(g)$ 
					 means that 
					$ h\cdot f(g)\cdot h^{-1}\cdot g^{-1}\neq e_G.$ It follows that $[u_1,u_2]$ is a nontrivial element of $G \leftthreetimes P.$
					This contradicts the fact that $N\cap (G\leftthreetimes P)$ it trivial. Therefore, all elements of $N$ have the form $(h,\gamma_h^{-1}).$ 
					The group $N$ is compact being a closed subgroup of the compact group $G\leftthreetimes (\overline{P}\cap \Inn(G)).$  
	 By the continuity of the canonical projection $pr\co G \leftthreetimes \overline{P} \to G$, we conclude that the subgroup $H:=pr(N)$ is closed (being compact) in $G.$				
					 Moreover, 
					 the structure of $N$ implies that $H$ is also nontrivial. 
					 
					 To  prove $(a)$ assume 
					 that there exists a nontrivial element  $h\in H\cap Z(G).$ But then $(h,\gamma_h^{-1})$ is a nontrivial element of $N\cap (G\leftthreetimes P),$ and that is a contradiction. 
					 
					 Property $(b)$ follows from the normality of $N.$ Indeed, if $f\in P$ and $h\in H$ we have  $(e_G,f)(h,\gamma_h^{-1})(e_G,f)^{-1}=(f(h),\gamma_{f(h)}^{-1})\in N.$ Hence, $f(h)\in H.$
				
					Finally, we prove property $(c).$
					Let $\gamma_h\in \Gamma(H),$ where $h\in H=pr(N).$ 
					Then, $(h,\gamma_h^{-1})\in N\subseteq G\leftthreetimes \overline{P}.$ Therefore, $\gamma_h^{-1}\in \overline P$ and since $\overline P$ is a group we also have $\gamma_h\in \overline P$. This proves that $\Gamma(H)\subseteq \overline P.$ 
						 Let us show that $\Gamma(H)\cap P$ is trivial. Otherwise, there exists a nontrivial $h\in H$ such that $\gamma_h^{-1}\in \Gamma(H)\cap P.$ But then  $(e_G,e_P)\neq(h, \gamma_h^{-1})\in N\cap (G\leftthreetimes P),$ contradicting the triviality of $N\cap (G\leftthreetimes P).$ This completes the proof.
				\end{proof}
				\begin{corol}\label{cor:min}
					Let $G$ be a compact group and $P\leq \Aut(G).$ Then $G\leftthreetimes P$ is minimal in each of the following cases:
					\begin{enumerate}
						\item $G$ is abelian;
						\item $P$ is essential in $\overline{P}$ 
						(e.g., $P$ is closed);
						\item 
						$P\cap \Inn(G)$ is essential in $\overline{P}\cap \Inn(G).$
					\end{enumerate}
				\end{corol}
				\begin{proof}
					In each case 
					at least one of the  conditions of Theorem \ref{thm:sol}.2 does not hold.\\ $(1)$: If $G$ is an abelian group then $G=Z(G).$ Thus for every nontrivial $H\leq G$ the group $H\cap Z(G)$ is nontrivial and so $(a)$ is impossible.\\
					$(2)$: Assume for a contradiction that $H$ is a nontrivial closed subgroup of $G$ satisfying conditions 
		$(a)-(c).$	It follows that $N=\Gamma(H)$ is a nontrivial closed subgroup of $\overline{P}$ which trivially intersects $P.$ This contradicts the assumption that $P$ is essential in $\overline{P}.$\\
					$(3):$  Assume for a contradiction that $H$ is a nontrivial closed subgroup of $G$ satisfying conditions 
			$(a)-(c). $
				 It follows that $N=\Gamma(H)$ is a nontrivial closed subgroup of $\overline{P}\cap \Inn(G)$ which trivially intersects $\overline{P}\cap \Inn(G).$ This contradicts the assumption that 
					$P\cap \Inn(G)$ is essential in $\overline{P}\cap \Inn(G).$
				\end{proof}
				\begin{thm}
					Let $G$ be a compact group with trivial center and $P\leq \Inn(G).$
					Then $G \leftthreetimes P$ is minimal if and only if $P$ is essential in $\overline{P}.$ 
				\end{thm}
				\begin{proof}
					Sufficiency follows from Corollary 
					\ref{cor:min}.3.\\
					To prove the necessity assume  that  $P$ is not essential in $\overline{P}.$ 
					Then, there exists a closed normal subgroup $N$ of $\overline{P}$ such that $N\cap P$ is trivial. Since the center of $G$ is trivial, the continuous 
					homomorphism $\Gamma\co G\to \Inn(G)$ is in fact a topological isomorphism. It follows that $H=\Gamma^{-1}(N)$ is a closed nontrivial subgroup of $G.$ Clearly, $H\cap Z(G)=\{e_G\}$. By the normality of $N,$ one can show that $H$ is $P$-invariant. Furthermore, $\Gamma(H)=N$ is a subgroup of $\overline{P}$ and $\Gamma(H)\cap P=N\cap P=\{e_P\}.$  It follows from 
					Theorem \ref{thm:sol} that $G \leftthreetimes P$ is not minimal.
				\end{proof}
				\begin{thm}\label{thm:bool}
					Let $G$ be a compact group containing a non-closed  Boolean subgroup $B$ and let $P=\Gamma (B).$  If in addition  $\overline{B}\cap Z(G)$ is trivial, then $G \leftthreetimes P$ is not minimal.
				\end{thm}
				\begin{proof}
					By Theorem \ref{thm:sol}, it suffices to show that there exists a closed nontrivial subgroup $H\leq G$ with properties $(a)-(c).$ 
					Since $B$ is a Boolean subgroup of $G$, its closure $\overline{B}$ is also Boolean. Fix $h\in \overline{B}\setminus B$ and let $H=\{h,e_G\}.$
					Then, $H$ is a nontrivial closed subgroup of $G$ with $H\cap Z(G)=\{e_G\}.$ For every $b\in B$  we have 
					$bhb^{-1}=h,$
					since $\overline{B}$ is abelian. This implies that $H$ is $P$-invariant. Clearly, $\Gamma(H)$ is a subgroup of $\overline{P}=\Gamma(\overline{B}).$ Finally, assume for contradiction that $\Gamma(H)\cap P\neq\{e_P\}.$ Hence, there exists $b\in B$ such that $\gamma_b=\gamma_h.$ 
					This implies that 
					 $e_G \neq hb\in \overline{B}\cap Z(G),$ a contradiction.
				\end{proof}
	We use Theorem \ref{thm:bool} in the following example, where we show that there exist a compact two-step nilpotent group $G$  and a subgroup  $P$ of $\Aut(G)$  such that
	$G\leftthreetimes P$ is not minimal. This answers a question of Dikranjan.   
	
				\begin{example}\label{nil}
					Let $R$ be the compact ring $\Z_2^{\N}$ (with operations defined coordinatewise) and 
					consider its dense subring 
					  $$\tilde R:=\{(x_n)_{n\in \N}:x_n\in \Z_2 \wedge \ |n:x_n\neq 0|<\infty\}.$$ 
					  Let $G:=(R\times R)\leftthreetimes R$ be the  {\it generalized Heisenberg group} (see, for example, \cite{MegDik}) defined via the action 	
					$$\pi \co R\times (R\times R) \to R\times R, \ \pi (f,(a,x))=(a+fx,x).$$
					 By \cite[Lemma 2.1]{Meg95},   $Z(G)=(R\times \{0_R\})\leftthreetimes \{0_R\}.$
					 Let $B:=(\{0_R\}\times \tilde R)\leftthreetimes \{0_R\}$  and 
					 let $P=\Gamma(B).$ Then, $B$ is a non-closed Boolean subgroup of $G.$ Indeed, its closure $\overline B$ coincides with $(\{0_R\}\times  R)\leftthreetimes \{0_R\}.$  Since  $\overline B\cap Z(G)$ is trivial, it follows from Theorem \ref{thm:bool} that $G\leftthreetimes P$ is not minimal.
				\end{example}

\vskip 0.3cm
	 \noindent {\bf Acknowledgments.} We thank
	 R. Ben-Ari  and
	 D. Dikranjan for valuable suggestions. We also thank the referee for the careful reading and 
	 very useful comments
	  and suggestions.

\bibliographystyle{plain}

\end{document}